\documentclass[11pt]{amsart}
\usepackage{amsmath}
  \usepackage{url,graphicx}
  \usepackage{amssymb, color, pstricks-add, manfnt}
  \usepackage{amsmath, amsthm,  amsfonts}
  \usepackage[plainpages=false]{hyperref}

  \theoremstyle{plain}
  \newtheorem{theorem}{Theorem}[section]
  \newtheorem{lemma}{Lemma}[section]
  \newtheorem{corollary}{Corollary}[section]

   \theoremstyle{remark}
  \newtheorem{remark}{Remark}[section]

  \numberwithin{equation}{section}
  \numberwithin{figure}{section}

\renewcommand{\baselinestretch}{1.00}
\begin{document}

\title{ A note on second derivative estimates for Monge-Amp\`ere type equations.}

\author{Neil S. Trudinger}
\address{Mathematical Sciences Institute, The Australian National University,
              Canberra ACT 0200, Australia}

\

\thanks{Research supported by Australian Research Council Grant (DP180100431)}




\maketitle

\abstract {In this note we revisit previous Pogorelov type interior and global  second derivative estimates of the author, F. Jiang and J. Liu for solutions of Monge-Amp\`ere type partial differential equations. Taking account of recent strict convexity regularity results of Guillen, Kitagawa and Rankin, and following our earlier work in the optimal transportation case, we remove the monotonicity assumptions in the more general case of generated Jacobian equations and consequently in the subsequent application to classical solvability and global regularity for second boundary value problems.}
\endabstract


\baselineskip=12.8pt
\parskip=3pt
\renewcommand{\baselinestretch}{1.38}

\section{Introduction}\label{Section 1}

\vskip10pt

In this note, we are concerned with Pogorelov type interior and global second derivative estimates of elliptic solutions of nonlinear partial differential equations of Monge-Amp\`ere type, (MATEs), which amplify and improve earlier results in \cite{LT2010, JT2014}and \cite{T2021}.  Such equations can be written in the general form, 

\begin{equation} \label{MATE}
\text{det}[ D^2u - A(\cdot,u,Du)] = B(\cdot,u,Du),
\end{equation}

\noindent 
\noindent where $A$ and $B$ are respectively $n \times n$ symmetric
matrix valued and scalar functions on a domain $\mathcal U\subset \mathbb{R}^n\times\mathbb{R}\times\mathbb{R}^n$ and $Du$ and $D^2u$ denote respectively the gradient and Hessian matrix of the scalar function $u\in C^2(\Omega)$, with one jet $J_1[u] (\Omega) \subset \mathcal U$, where  $\Omega$ is a bounded domain in $\mathbb{R}^n$.  A solution $u\in C^2(\Omega)$ of
(\ref{MATE}) is called  elliptic, (degenerate elliptic), whenever $ w = D^2u-A(\cdot,u,Du)>0, (\ge 0)$, which implies $B>0, (\ge0)$. We assume throughout that
$A$ and $B$ are $C^2$ smooth, with $B > 0$,  and that the matrix function $A$ is \emph{regular} with respect to the gradient variables, that is, denoting points in $\mathcal U$ by $(x,z,p)$, 

\begin{equation}\label{1.2}
A^{kl}_{ij}\xi_i\xi_j\eta_k\eta_l: = (D_{p_kp_l}A_{ij}) \xi_i\xi_j\eta_k\eta_l \ge 0
\end{equation}

\noindent in $\mathcal U$, for all  $\xi,\eta \in \mathbb{R}^n$ such that $\xi \!\cdot\! \eta = 0$.

For convenience, we will also assume throughout that $\mathcal U$ is convex in $p$ for fixed $x$ and $u$. Condition \eqref{1.2}, which originated in its strict form A3 for regularity in the special case of optimal transportation in \cite {MTW2005},  corresponds to the weak form A3w introduced for global regularity in \cite{Tru2006, TW2009} and was shown to be sharp in \cite{Loe2009}. When the matrix function $A$ is not assumed twice differentiable, we may more generally express its regularity as  convexity of the form $A\xi.\xi$ along straight line segments in $p$, othogonal to $\xi$, which also suffices for the associated optimal transportation convexity theory and the more general convexity theory of generating functions \cite{LT2020,LT2021}.

Following our note on optimal transportation regularity \cite{T2013}, we can then remove the monotonicity conditions in our second derivative estimates for classical solutions of generated Jacobian equations  in \cite {JT2014}, under an appropriate local strict convexity control, thereby providing a corresponding extension of the classical existence result in Theorem 1.1 in \cite{JT2018}. Paralleling \cite{T2013},  the crucial elements in our approach are an extension of the Pogorelov type estimate in \cite{LT2010}, using Lemma 3.3 in \cite{T2021}, and the strict convexity result in \cite{GK2017}. These results have been flagged in \cite {T2021} and can also be  further improved using recent work by Rankin \cite {Ra2021}. 

We will treat  the Pogorelov estimates in Section 2, followed by their application to second derivative bounds for solutions of the second boundary value problem for generated Jacobian equations in Section 3. Finally in Section 4 we will consider the application to the existence of globally smooth classical solutions, thereby removing the monotonicity conditions on the matrix function $A$ in \cite{JT2018, Ra2021}. 

\vskip10pt

\section{Pogorelov estimates}\label{Section 2}

We begin with an interior Pogorelov estimate which combines  those in \cite{JT2014, LT2010, T2021}. In its formulation we use  the linearized operator $\mathcal L$ defined by 

\begin{equation}\label{linearized operator}
\mathcal L= \mathcal L[u] : = w^{ij}[D_{ij}-D_{p_k}A_{ij}(\cdot,u,Du)D_k] ,
\end{equation}

\noindent where  $[w^{ij}]$ denotes the inverse of $w= [w_{ij}]$.

\begin{theorem}
Let $u\in C^4(\Omega)\cap C^{0,1}(\bar \Omega)$ be an elliptic solution of \eqref{MATE} in $\Omega$,  and let $u_0 \in C^2(\Omega)\cap C^{0,1} (\bar\Omega)$ be a degenerate elliptic supersolution such that  $J_1[u], J_1[u_0] \subset \mathcal U_0 \subset\subset \mathcal U$  and $u=u_0$ on $\partial\Omega$, $u < u_0$ in $\Omega$. Suppose there exists a barrier function $\phi \ge 0, \in C^2(\bar\Omega)$ satisfying 
\begin{equation}\label{barrier}
\mathcal L\phi \ge w^{ii} -C^\prime
\end{equation}

\noindent in $\Omega$, for some constant $C^\prime \ge 0$. Then there exist positive constants $\beta$, $\delta_1$, $\delta_2$ and $C$ depending on $n, \mathcal U, \mathcal U_0, |A, \log B|_{2,\mathcal U_0}, C^\prime$ and $|\phi|_{1;\Omega}$ such that, if either (i) $D_u A > -\delta_1 I$ in $\mathcal U_0$ or (ii) $\text{diam}\ \Omega < \delta_2$, then 

\begin{equation} \label{estimate}
\sup_\Omega (u_0 -u)^\beta |D^2u| \le C.
\end{equation}
\end{theorem}

\begin{proof}

First we note that Case (ii) is proved in \cite {T2021}, Lemma 3.3 and  the barrier hypothesis \eqref{barrier} is automatically satisfied with $C^\prime = 0$, for suitable $\delta_2$, by taking $\phi(x) =$ const.$|x-x_1|^2$  for some point $x_1\in\Omega$.  Accordingly we will concentrate on Case (i). Adapting the proofs of Lemma 3.3 in \cite{T2021}, Theorem 1.2 in \cite{JT2014} and Theorem 2.1 in \cite {LT2010}, we first consider an auxiliary function,

\begin{equation}
v=v(\cdot, \xi)=\log(w_{ij}\xi_i\xi_j) +\tau |Du|^2+ e^{\kappa\phi} + \beta \log(u_0 -u)
\end{equation}

\noindent  where $|\xi| =1$ and  $\tau$, $\kappa$, $\beta$ are positive constants to be chosen.  Then we obtain, in place of inequality (3.9) in \cite{T2021}, at a maximum point $x_0$ and vector $\xi=e_1$ of $v$ in $\Omega$,

\begin{equation}\label{2.5}
Lv\ge \tau w_{ii} +\frac{\kappa}{2}(1+\kappa |D_i\phi|^2) e^{\kappa\phi} w^{ii}  - C(\tau+\kappa e^{\kappa\phi}) + \frac{1}{2w_{11}^2}\sum_{i>1}w^{ii}(D_iw_{11})^2 - \frac{C\beta}{\eta} -  \frac{C\beta}{\eta^2}w^{ii}|D_i\eta|(\eta+|D_i\eta|),
\end{equation}

\noindent provided $\tau \ge C$ and $\kappa \ge C(\tau + \beta\delta_1)$, where $C$ is a constant depending on the same quantities as in the estimate \eqref{estimate} and $L = \mathcal L  - D_{p_k}\log B(\cdot,u,Du)D_k, \eta = u_0 - u$. 

From the condition  $Dv(x_0) = 0$, we then have, for each  $i=1, \cdots n$,

\begin {equation} \label{2.6}
\beta\frac{|D_i\eta|}{\eta} \le \frac{|D_iw_{11}|}{w_{11}} + C \tau (w_{ii} + 1) + \kappa e^{\kappa\phi} |D_i \phi |
\end{equation}

\noindent so that for each $i>1$, 

\begin {equation} \label{2.7}
\beta \left(\frac{D_i\eta}{\eta}\right)^2 \le \frac{2}{\beta w_{11}^2} \left(D_iw_{11}\right)^2
       + C\frac{\tau^2}{\beta}(w_{ii}^2+1)+\frac{3\kappa^2}{\beta} e^{2\kappa\phi} (D_i \phi )^2.
\end{equation}

Assuming $\eta w_{11} (x_0) \ge \beta$  and combining \eqref{2.5}, \eqref{2.6}, \eqref{2.7}, we then obtain $Lv(x_0) >0$, by choosing $\tau \ge C$, $\kappa \ge C\tau$ , $\beta\ge C(\tau^2 + e^{\kappa\max\phi})$ and $\delta_1 \le \frac{1}{\beta}$ for sufficiently large  constant $C$ depending on $n, \mathcal U, \mathcal U_0, |A, \log B|_{2, \mathcal U}$ and $\phi$ and hence conclude the estimate \eqref{estimate}. 

\end{proof}




For generated Jacobian equations, in \cite{ JT2014}, we construct barriers satisfying \eqref{barrier} when the matrix function $A$ is either non-decreasing or non-increasing with respect to $u$; (see also \cite{JT-Oblique II}). This then includes the special case of optimal transportation equations in \cite{LT2010}, when $A$ is independent of $u$. We also remark that even though we have essentially used the same auxiliary functions as in the proofs of the Pogorelov estimates, Theorem 1.1 in \cite{LT2010} and Theorem 1.2  in \cite{JT2014}, the proof details are somewhat different and our proof here can be seen as clarifying the approaches in those papers. Recently, in the proof of Theorem 1.1 in \cite{JT2021},  we have also adapted the exponentiation of the barrier $\phi$ in this proof  to provide a correction to the proof of second derivative bounds for the Neumann problem in \cite{JTX2016}.  Moreover, we may also, by adapting the extension to the degenerate case in Theorem 1.2 in \cite{JT2021}, express the dependence on $B$ in Theorem 2.1 in terms of $\sup_{\mathcal U_0} B$ and the constants, $C_1, C_2 $ and $C_3$, defined respectively in equations (1.11), (1.12), and (1.13) in \cite{JT2021}.

From Case (ii) in Theorem 2.1 or Lemma 3.3  in \cite{T2021}, we can now infer interior and global second derivative bounds elliptic solutions of generated Jacobian equations satisfying appropriate local strict convexity conditions with only condition (1.2) assumed on the function $A$. Let us recall from \cite{T2014, JT2014}, that equation \eqref{MATE} is a generated Jacobian equation if $A$ is determined by a generating function $g\in C^2(\Gamma)$ for some domain $\Gamma \subset \mathbb{R}^n\times\mathbb{R}^n\times\mathbb{R}$, whose projections $$I(x,y) = \{z\in\mathbb{R} |\  (x,y,z)\in\Gamma\}$$ are open intervals.  Denoting points in $\Gamma$ by $(x,y,z)$, and 
 
 $$\mathcal{U}=\{(x,g(x,y,z),g_x(x,y,z))|\  (x,y,z)\in \Gamma\},$$ 
 
 \noindent this means that $g_z < 0$ in $\Gamma$ and there exist unique $C^1$ mappings $Y, Z: \mathcal U \rightarrow \mathbb{R}^n, \mathbb{R}$ satisfying
 
\begin{equation}\label{generating equation}
g(x,Y,Z)=u, \quad g_x(x,Y,Z)=p,
\end{equation}

\noindent with the matrix $A$ given by 
\begin{equation}\label{A}
A(\cdot,u,p)=g_{xx}(\cdot,Y(\cdot,u,p),Z(\cdot,u,p)). 
\end{equation}

Now suppose $u\in C^2(\Omega)$, $J_1[u] (\Omega) \subset \mathcal U$ and $u$ is elliptic in $\Omega$. Then $u$ is locally strictly $g$-convex in 
$\Omega$, in the sense that for each $x_0 \in \Omega$, there exists a $g$-affine function $g_0 = g(\cdot, y_0,z_0)$, with $y_0 = Y(\cdot,u,Du)(x_0)$,
$z_0 = Z(\cdot,u,Du)(x_0)$, such that $g(x_0) = u(x_0)$ and $g_0 < u$ in $\mathcal N_0 -\{x_0\}$, for some neighbourhood $\mathcal N_0$ of $x_0$.
For $h>0$, we can define the section 

$$ S_h(u,g_0) = \{ u\in\Omega | (x,y_0,z_0 -h) \in \Gamma,\  u(x) < g_0(x,y_0,z_0-h) \}, $$

\noindent and let $S_h(u,g_0,x_0)$ denote the component of $S_h(u,g_0)$ containing $x_0$.  For $0<R< \text{dist}(x_0, \partial\Omega)$, we can then define a modulus of strict $g$-convexity of $u$ at $x_0$ by 

$$ \omega(R) = \omega[u](x_0,R) := \sup\{h | S_h(u,g_0,x_0) \subset B_R(x_0)\}. $$

From Lemma 3.3 in \cite {T2021}, or Corollary 2.1, we then have the following interior second derivative estimate for elliptic solutions of generated Jacobian equations. In it's formulation we use the quantity

\begin{equation} \label {convexity modulus}
 \omega[u](\Omega^\prime, \Omega) = \inf_{x_0\in \Omega^\prime, R< R_0} \omega[u](x_0,R) 
 \end{equation}

\noindent to denote a lower bound on the modulus of strict $g$-convexity of $u$ over a subdomain $\Omega^\prime \subset\subset \Omega$, where $R_0 = $dist$(\Omega^\prime,\partial\Omega)$.

\begin{theorem}
Let $u\in C^4(\Omega)\cap C^{0,1}(\bar \Omega)$ be an elliptic solution of  a generated Jacobian equation \eqref{MATE} in the domain $\Omega$, with  $J_1[u] \subset \mathcal U_0 \subset\subset \mathcal U$. Then, for any strictly contained subdomain $\Omega^\prime$, we have the estimate

\begin{equation} \label{interior estimate}
\sup_{\Omega^\prime} |D^2u| \le C,
\end{equation}

\noindent where $C$ depends on $n,\mathcal U, \mathcal U_0, g, B, R_0$ and $\omega[u](\Omega^\prime, \Omega)$.

\end{theorem}

If some neighbourhood, $\mathcal N^\prime$, of $\partial\Omega$ is $A$-bounded, we can then infer a global second derivative bound from Theorem 2.2, using the global estimate, Theorem 3.1, in \cite {TW2009}.  Here we recall the definition from  \cite{Tru2006, LT2010} that a domain $\Omega$ is $A$-bounded with respect to $u$ if there exists a barrier function $\phi\in C^2(\bar\Omega)$ satisfying 
\begin{equation}
\big{[} D_{ij} \phi  - D_pA_{ij}(\cdot,u,Du)\big{]} \xi_i \xi_j \ge |\xi|^2,
\end{equation}
\noindent in $\Omega$, for all $\xi \in \mathbb{R}^n$. Clearly, if $\Omega$ is $A$-bounded with respect to $u$ then the barrier condition \eqref{barrier} is satisfied with $C^\prime = 0$.

\begin{corollary}
Let $u\in C^4(\Omega)\cap C^2(\bar \Omega)$ be an elliptic solution of  a generated Jacobian equation \eqref{MATE} in the domain $\Omega$, 
$J_1[u] \subset \mathcal U_0 \subset\subset \mathcal U$, and suppose $\mathcal N^\prime = \Omega - \bar\Omega^\prime$ is $A$-bounded, with respect to $u$, for some subdomain $\Omega^\prime \subset\subset \Omega$, with barrier $\phi$. Then we have the estimate,

\begin{equation} \label{global estimate}
\sup_\Omega  |D^2u| \le C ( 1+ \sup_{\partial\Omega} |D^2u|),
\end{equation}

\noindent where  $C$ depends on $n,\mathcal U, \mathcal U_0, g, B, R_0$, $\omega[u](\Omega^\prime, \Omega)$  and $\phi$.

\end{corollary} 

We will apply Corollary 2.1 in the next section to obtain global second derivative estimates for solutions of the second boundary value problem for generated Jacobian equations, complementing those in \cite{JT2014}.  Here we just note that an appropriate condition for $A$-boundedness of some
$\mathcal N^\prime$ is the uniform $A$-convexity of $\Omega$ with respect to $u$, as defined in \cite{Tru2006, JT2014},  which is also a critical condition for estimating  $D^2u$ on $\partial\Omega$, in \cite {JT2014}.

\vskip10pt

\section{Second boundary value problem}\label{Section 3}

First, we recall that a generated Jacobian equation is a special case of  a prescribed Jacobian equation,

\begin{equation}\label{PJE}
\det DY(\cdot, u, Du)=\psi(\cdot, u, Du),
\end{equation}

\noindent where $\psi$ is a given scalar function on $\mathcal U$.  In this case the scalar function $B$ in \eqref{MATE} is given by 

\begin{equation} \label{B}
B(\cdot,u,p)=\det E(\cdot,Y(\cdot,u,p),Z(\cdot,u,p))\psi(\cdot,u,p). 
\end{equation}

\noindent where the matrix function $E$, given by

$$E = [E_{i,j}] =  g_{x,y} - (g_z)^{-1}g_{x,z}\otimes g_y.$$
satisfies $\det E \ne 0$.

The second, (or natural), boundary value problem for prescribed Jacobian equations is to prescribe the image
\begin{equation}\label{Second boundary problem}
Tu(\Omega):=Y(\cdot,u,Du)(\Omega)=\Omega^*,
\end{equation}
where $\Omega^*\subset\mathbb{R}^n$ is a target domain.  

So far our assumptions on the generating function $g$ correspond to conditions A1, A2 and A3w in \cite{JT2014}. To prove second derivative estimates for solutions of \eqref{MATE},\eqref{Second boundary problem}, we also assume the dual condition A1$^*$, namely that the mapping $Q: = -g_y/g_z$ is one-to-one in $x$, for all $(x,y,z) \in \Gamma$. We also recall the notion of dual generating function $g^*$, defined on the dual set 
$$\Gamma^* :=\{(x,y,g(x,y,z)) |  (x,y,z) \in \Gamma\}$$
by $$g(x,y,g^*(x,y,u))=u.$$ 

As in \cite{JT2018}, it will be convenient to formulate our second derivative bounds using domain convexity assumptions expressed in terms of the mapping $Y$. Namely, for an open interval $J$, satisfying  $\Omega \times \Omega^* \times J \subset\subset \Gamma^*$, we repeat the following definitions from \cite{JT2018}.

The $C^2$ domain $\Omega$ is $Y$-convex (uniformly $Y$-convex) with respect to $\Omega^*\times J$ if it is connected and
\begin{equation}\label{Y-convex}
        [D_i\gamma_j(x)- D_{p_k}A_{ij}(x,u_0,p)\gamma_k(x)]\tau_i\tau_j\ge 0,(\delta_0),
    \end{equation}
for all $x\in\partial\Omega$, $u_0\in {J}$, $Y(x,u_0,p) \in \Omega^*$, unit outer normal $\gamma$ and unit tangent vector $\tau$, (for some constant $\delta_0>0$).

The domain $\Omega^*$ is $Y^*$-convex (uniformly $Y^*$-convex) with respect to $\Omega\times  J$
if  the images
$$\mathcal P(x,u_0,\Omega^*) = \{p\in\mathbb{R}^n | \ (x,u_0,p)\in \mathcal U, Y(x,u_0,p)\in\Omega^*\}$$
are convex for all $(x,u_0)\in \Omega\times J$, (uniformly convex for all $x\in\overline\Omega$, $u\in\overline J$).

As remarked in \cite {JT2018},  $Y^*$-convexity is equivalent to the notion of $g^*$-convexity, introduced in \cite{T2014}, while $Y$-convexity is implied by $g$-convexity with respect to $y\in \Omega^*$ and $z =g^*( x,y,u_0)$ for all $x\in \Omega$ and $u_0\in J$. Moreover, under this stronger condition, elliptic solutions $u$ of the boundary value problem \eqref{MATE}, \eqref{Second boundary problem}, satisfying $u(\Omega) \subset\subset J$, will be globally strictly $g$-convex in $\Omega$ \cite{T2021}. 

By combining Corollary 2.2 with Lemma 3.2 in \cite{JT2014} and Theorem 2.3 in \cite{GK2017}, or more specifically Theorem 1 in \cite{Ra2021}, we now obtain the following complimenting estimate to Theorem 3.1 in \cite{JT2014}. In its formulation and applications, it will be convenient to fix an open interval $J$, satisfying $\Omega \times \Omega^* \times J \subset\subset \Gamma^*$.

\begin{theorem}\label{Th3.1}
Let $u\in C^4(\bar\Omega)$, with $J_1[u] \subset \mathcal U_0 \subset\subset \mathcal U$, $u(\Omega) \subset J_0$, for some open interval $J_0 \subset J$, be an elliptic solution
of the second boundary value problem \eqref{MATE}, \eqref{Second boundary problem}  in $\Omega$, where  $A\in
C^2(\mathcal U)$ is given by \eqref{A} with generating function $g\in C^4(\Gamma)$, satisfying
conditions A1, A2, A1* and A3w,  $B >0,\in
C^2(\mathcal U)$ and the domains $\Omega, \Omega^* \in C^4$ are respectively
uniformly $Y$-convex and uniformly $Y^*$-convex with respect to $\Omega^*\times J_0$ and $\Omega\times J_0$. Suppose additionally that (i) $\Omega \subset \subset \Omega_0$ for a domain $\Omega_0$, also satisfying $\Omega_0\times \Omega^* \times J \subset\subset \Gamma^*$, which is $g$-convex with respect to $y\in \Omega^*$ and $z =g^*( x,y,u_0)$ for all $x\in \Omega$, $u_0\in J_0$ and (ii) $Tu$ is one-to-one and $u$ is $g$-convex in $\Omega$, with any $g$-support, $g_0$, satisfying $g_0(\Omega) \subset J$. Then we have the estimate,
\begin{equation}\label{bound}
\sup\limits_\Omega |D^2u| \le C,
\end{equation}
where the constant $C$ depends on $n, \mathcal U, \mathcal U_0, g, B, \Omega, \Omega^*,  \Omega_0, J_0$ and $J$.
\end{theorem}

\begin{proof}

Since the proof is a straightforward extension of that of optimal transportation case in Theorem 2.1 in \cite{T2013}, we just describe it  briefly here. From Lemma 3.2 in \cite{JT2014} and Theorem 2.2, it is enough to prove an estimate from below for the modulus of strict $g$-convexity of of $u$ over any subdomain $\Omega^\prime \subset\subset\Omega$, as formulated in \eqref{convexity modulus}. Defining $\psi$ through equation\eqref{B}, we first note that a $g$-convex, elliptic solution $u\in C^2(\Omega)\cap C^1(\bar\Omega)$ of the second boundary value problem  \eqref{MATE}, \eqref{Second boundary problem}, for which $Tu$ is one-to-one, will be a generalized solution, as defined in Section 4 of \cite{T2014}, with density $f =|\psi|(\cdot,u,Du)$, satisfying $\int_\Omega f = |\Omega^*|$, (and target density $f^*\equiv 1$). Furthermore there exist positive constants $c_0$ and $c_1$, depending on $\psi, \mathcal U, \mathcal U_0$ such that $c_0 \le f \le c_1$. Consequently, if there does not exist a lower bound for $\omega[u](\Omega^\prime, \Omega)$, by virtue of the weak continuity of the associated measures, there would exist a sequence of such solutions converging uniformly to a generalized solution $u$, with $L^\infty$ density $f$ satisfying the same bounds, which is not strictly $g$-convex at some point $x_0\in \bar\Omega^\prime$, thereby contradicting Theorem 1 in \cite{Ra2021} and  Theorem 2.3 in \cite{GK2017}. 

Alternatively, we remark that we can obtain an explicit estimate for $\omega[u]$ from the H\"older gradient estimate in Section 8 of \cite{GK2017} corresponding to Theorem 2.4 there, by using duality.
\end{proof}

From Theorem 3.1 in \cite{JT2014},  we note that the additional conditions (i) and (ii) are not needed if any of conditions A3, A4w or A4$^*$w are also satisfied, or more generally either A3 or the existence of  a  barrier $\phi$ satisfying \eqref{barrier}. As in \cite{JT2014}, we also have a stronger version of Theorem 3.1, using the full strengths of Lemma 2.2 in \cite{JT2014} and Theorem 1 in \cite{Ra2021}. Namely, we need only assume in our domain convexity conditions that $\Omega$ and $\Omega^*$ are respectively uniformly $Y$-convex and $Y^*$-convex with respect to $(\Omega^*, u)$ and $(\Omega, u)$, as defined in \cite{JT2014}, while $\Omega_0$ is $g$-convex with respect to $y\in\Omega^*$ and $z=g^*(x,y,u(x))$, for all $x\in \Omega$, $y\in\Omega^*$. Moreover, taking account of the injectivity of $Tu$, the $Y$-convexity conditions are equivalent to the uniform $g$-convexity of $\Omega$ with respect to the dual function $v = u^*_g$ on $\Omega^*$ and the uniform $g^*$-convexity of $\Omega^*$ with respect to the function $u$ itself on $\Omega$. The reader is referred to \cite{T2014} and \cite{JT2014} for further details concerning our notions of domain convexity.

\vskip10pt

\section{Application to existence and regularity} \label{section 4}

For our applications to  existence, we will assume the function $\psi$ is separable in the sense that 
\begin{equation}\label{psi}
|\psi|(x,u,p)=\frac{f(x)}{f^*\circ Y(x,u,p)},
\end{equation}
for positive intensities $f\in L^1(\Omega)$ and $f^*\in L^1(\Omega^*)$. Then a necessary condition for the existence of an elliptic solution with the mapping $Tu$ being a diffeomorphism, to the second boundary value problem \eqref{MATE}, \eqref{Second boundary problem}, is the conservation of energy
\begin{equation}\label{conservation of energy}
\int_\Omega f =\int_{\Omega^*} f^*.
\end{equation}  

To fit our previous conditions on $B$ we will assume the functions $f$ and $f^*$ are both $C^2$ smooth, with positive lower and upper bounds. 

It then follows from Theorem 3.1, that our classical existence result, Theorem 1.1 in \cite{JT2018}, holds without assuming any of the conditions A3, A4w and A4$^*$w. Moreover, from \cite{Ra2020, Ra2021}, this result can be refined in the sense that for any $x_0\in \Omega$ and $u_0\in J_0$, sufficiently far from the boundary of the gradient control interval $J_0$ in condition A5, there exists a unique $g$-convex, uniformly elliptic solution $u$, of the second boundary value problem \eqref{PJE}, \eqref{Second boundary problem}, satisfying $u(x_0) = u_0$. We will now formulate these extensions more explicitly, with the interval $J_0$ also permitted to be finite. First we repeat the formulation of condition A5 in the finite case.

\begin{itemize}
\item[{\bf A5}:]
There exists  an open interval $J_0 =(m_0, M_0)\subset J$ 
and a positive constant $K_0$, such that 
$$  |g_x(x,y,z)| < K_0, $$
for all $x\in\bar\Omega, y\in \bar\Omega^*, g(x,y,z) \in J_0$.
\end{itemize}

Then we have the following extension of the classical existence results in \cite{JT2018, Ra2021}. Taking account of condition (i) in Theorem 3.1, it will be convenient in its formulation to use domain convexity with respect to the generating function $g$, as in \cite{Ra2021}, rather than the mapping $Y$, and treat the slightly more general situation in a remark. 

\begin{theorem}\label{Th4.1}
Let $g\in C^4(\Gamma)$ be a generating function satisfying conditions A1, A2, A1*, A3w, and A5, for $C^4$ bounded domains
$\Omega$, $\Omega^*$ in $\mathbb{R}^n$, with $\Omega$ uniformly $g$-convex with respect to $y\in\bar\Omega^*$, $z\in g^*(\bar\Omega,y, \bar J{_0})$  and $\Omega^*$  uniformly $g^*$-convex with respect to $x\in\bar\Omega$, $u\in \bar J_0$. Also assume there exists a $g$-affine function, $g_0 = g(\cdot, y_0, z_0)$, on $\bar\Omega$ satisfying $Tg_0 = y_0 \in \Omega^*$,
\begin{equation}\label{g_0}
g_0(\bar\Omega) \subset (m_0+K_0d, M_0 - K_0d),
\end{equation} 
and
\begin{equation}\label{4.4}
 g(\Omega,y,z) \subset J,\ \text{for all} \ y\in \Omega^*,  \ z\in g^*(\Omega,y, (m_0, m_0 + K_0d)), 
 \end{equation}
\noindent  where $d=\text{diam}(\Omega)$.  Suppose also the function $\psi$ satisfies \eqref{psi}, \eqref{conservation of energy}. Then for any $x_0 \in \Omega$, there exists a unique $g$-convex elliptic solution $u\in C^3(\bar \Omega)$ of the second boundary value problem \eqref{PJE}, \eqref{Second boundary problem}, satisfying $u(x_0) = g_0(x_0)$.  Furthermore, the mapping $Tu$ is a $C^2$ smooth diffeomorphism from $\bar\Omega$ to $\bar \Omega^*$.
\end{theorem}

\begin{remark}\label{Remark 4.1}
To avoid possible confusion, we point out that here, as in condition (2.3) in \cite{JT2014} and Theorem 1.1 in  \cite{T2021}, we are using the following meaning of the diameter of of a domain $\Omega \subset \mathbb{R}^n$,

$$ \text{diam}(\Omega) = \sup_{x,y \in \Omega} d_{\Omega} (x,y), $$

\noindent where $d_{\Omega} (x,y) = \text{dist}_\Omega(x,y)$ is the distance in $\Omega$ between $x$ and $y$, that is the infimum of the lengths of $C^1$ curves in $\Omega$ joining the points $x$ and $y$. Moreover, we may also refine condition \eqref{g_0} by replacing $d$ by 
$d^\prime:= \text{diam}(\Omega^\prime)$, where $\Omega^\prime$ is any larger domain than $\Omega$, also satisfying condition A5. If 
$\Omega^\prime = \hat\Omega$ is the convex hull of $\Omega$, we obtain the usual notion of diameter in $\mathbb{R}^n$. In fact, the above confusion goes back to the statement of the existence result for generalized solutions in Theorem 4.2 in \cite{T2014}, where the corresponding distance in $\Omega$ is also intended. 
\end{remark}

\begin{proof} 
Substituting the second derivative estimate in Theorem 3.1 for that in Theorem 3.1 in \cite{JT2014}, we would infer from the proof of Theorem 1.1 and Remark 3.2 in \cite{JT2018}, the existence of a solution $u$ whose graph intersects that of $g_0$. For this we need to observe that the assumed uniform $g$-convexity conditions on  $\Omega$ would imply the corresponding uniform $Y$-convexity of $\Omega$ as well as the supplementary condition (i) and the $g$-convexity of an elliptic solution $u$, from Lemma 2.1 in \cite{T2021}. Then the rest of the supplementary condition (ii) now follows from the mass balance condition \eqref{conservation of energy} and our assumed condition \eqref{4.4}. To get the full strength of Theorem 4.1 we then need to adjust the homotopy family (3.4) in \cite{JT2018}, as done by Rankin in equation (93) in \cite{Ra2021}, to conclude the existence of a solution $u$, with graph intersecting that of $g_0$ at $x_0$, the uniqueness of which then follows from \cite{Ra2020}. 
\end{proof}

\begin{remark}\label{Remark 4.2}
We can write a cleaner but slightly weaker version of Theorem 4.1 by replacing the uniform $g$-convexity condition on $\Omega$ by uniform 
$Y$-convexity as in Theorem 1.1 of \cite{JT2018} with our condition on $g_0$, \eqref{g_0}, strengthened to 
\begin{equation}\label{g_0w}
g_0(\bar\Omega) \subset (m_0+2K_0d, M_0 - 2K_0d),
\end{equation} 
which also implies condition \eqref{4.4}.  As mentioned above, we also have a messier, more general statement if we replace the uniform $g$-convexity of  $\Omega$ by uniform $Y$-convexity, but still assume the corresponding $g$-convexity of $\Omega$ and supplementary condition (i) in Theorem 3.1.
Note that when we use Theorem 3.1 in \cite{JT2014} in the proof of Theorem 4.1, we do not need this last condition when any of the conditions A3, A4w and A4$^*$w hold.  \end{remark}

From the proof of Theorem 2 in \cite{Ra2021}, which combines the existence of classical solutions with the local regularity of strictly convex generalized solutions from \cite{T2021} and the uniqueness from \cite{Ra2020}, we then have the following global regularity result, which extends Theorem 2 in \cite{Ra2021} to the case when A4w is not assumed.

\begin{corollary} Let $u$ be a $g$-convex generalized solution of the second boundary value problem, \eqref{PJE}, \eqref{Second boundary problem}, where $g,\Omega, \Omega^*$ and $\psi$ satisfy the hypotheses of Theorem 4.1, with $g_0$ a $g$-support of $u$ at some point $x_0\in\Omega$. Then
$u\in C^3(\bar\Omega)$ is a classical elliptic solution of \eqref{PJE}, \eqref{Second boundary problem} and $Tu$ a $C^2$ smooth diffeomorphism from $\bar\Omega$ to $\bar \Omega^*$.
\end{corollary}

\begin{remark}\label{Remark 4.3}
The modified hypotheses of Theorem 4.1 in Remark 4.2 are also applicable to Corollary 4.1. Moreover, if any of conditions A3, A4w and A4$^*$w hold, we need only assume, in accordance with Theorem 2 in \cite{Ra2021}, that (4.4) only holds for any $g$-support of $u$. By adapting the uniqueness argument for global optimal transportation regularity in Section 6 of \cite{TW2009}, we may remove condition (4.4) completely in these cases if also 

$$\sup_\Omega u < M_0 - 2K_0d.$$


\noindent The overall proof is also much simpler in this case in that we not need to use the strict convexity of generalized solutions and their local regularity. A similar remark applies to the general case in Corollary 4.1, except we still need condition (4.4) for the strict convexity control used in our proof of the second derivative estimates in Theorem 3.1.
\end{remark}

\vskip10pt



\baselineskip=12pt
\parskip=0pt

\begin{thebibliography}{}
%
%







\bibitem{GK2017} Guillen, N., Kitagawa, J.: Pointwise inequalities in geometric optics and other generated Jacobian equations. Comm. Pure Appl. Math. \textbf{70}, 1146-1220 (2017)

\bibitem{JT2014} Jiang, F., Trudinger, N.S.: On Pogorelov estimates in optimal transportation and geometric optics. Bull. Math. Sci. \textbf{4}, 407-431 (2014)

\bibitem{JT2018} Jiang, F., Trudinger, N.S.: On the second boundary value problem for Monge-Amp\`ere type equations and geometric optics. Arch. Rat. Mech.Anal. \textbf{229}, 547-567 (2018)

\bibitem{JT-Oblique II} Jiang, F., Trudinger, N.S.: Oblique boundary value problems for augmented Hessian equations II. Nonlinear Anal. \textbf{154}, 148-173 (2017)

\bibitem{JT2021} Jiang, F., Trudinger, N.S.: On the Neumann problem for Monge-Amp\`ere type equations revisited, New Zealand Journal of Mathematics \textbf{52}, 671-689, (2021)


\bibitem{JTX2016} F. Jiang, N.S. Trudinger, N. Xiang, On the Neumann problem for Monge-Amp\`ere type equations, Canadian Journal of Mathematics, 68, 1334-1361, (2016)




\bibitem{LT2010} Liu, J., Trudinger, N.S.: On Pogorelov estimates for Monge-Amp\`ere type equations. Discrete Contin. Dyn. Syst. \textbf{28}, 1121-1135 (2010)

\bibitem{LT2016} Liu, J., Trudinger, N.S.: On classical solutions of near field reflection problems. Discrete Contin. Dyn. Syst. \textbf{36}, 895-916 (2016)

\bibitem{Loe2009} Loeper, G.: On the regularity of solutions of optimal transportation problems. Acta Math. \textbf{202}, 241-283 (2009)

\bibitem{LT2020} Loeper, G., Trudinger, N.S.: Weak formulation of the MTW condition and convexity properties of potentials, Methods Appl. Anal. \textbf{28}, 53-60 (2021)

\bibitem{LT2021} Loeper, G., Trudinger, N.S.:  On the convexity theory of generating functions, preprint, arXiv: 2109.04585 (2021)

\bibitem{MTW2005} Ma, X.-N., Trudinger N.S., Wang, X.-J.: Regularity of potential functions of the optimal transportation problem. Arch. Ration. Mech. Anal. \textbf{177}, 151-183 (2005)

\bibitem{Ra2020} Rankin, C.: Distinct solutions to generated Jacobian equations cannot intersect,  Bull. Aust. Math. Soc. \textbf{102}, 462-470 (2020)

\bibitem{Ra2021} Rankin, C.: Strict $g$-convexity for generated Jacobian equations with applications to global regularity, arXiv:2111.00448 (2021)

\bibitem{Tru2006}  Trudinger, N.S.: Recent developments in elliptic partial differential equations of Monge-Amp\`ere type. ICM. Madrid, \textbf{3}, 291-302 (2006)


\bibitem{T2013} Trudinger, N.S.: A note on global regularity in optimal transportation. Bull. Math. Sci. \textbf{3}, 551-557 (2013)

\bibitem{T2014} Trudinger, N.S.: On the local theory of prescribed Jacobian equations. Discrete Contin. Dyn. Syst. \textbf{34}, 1663-1681 (2014)

\bibitem{T2021} Trudinger, N.S.: On the local theory of prescribed Jacobian equations revisited, Mathematics in Engineering, \textbf{3}, 1-17 (2021)



\bibitem{TW2009} Trudinger, N.S., Wang, X.-J.: On the second boundary value problem for Monge-Amp\`ere type equations and optimal transportation. Ann. Scuola Norm. Sup. Pisa Cl. Sci. \textbf{VIII}, 143-174 (2009)


\end{thebibliography}
\end{document}